\documentclass{amsart}
\usepackage{amsfonts}
\usepackage{amssymb,amsmath,ams}

\newtheorem{theorem}{Theorem}[section]
\newtheorem{lemma}[theorem]{Lemma}
\newtheorem{proposition}[theorem]{Proposition}

\newtheorem{remark}[theorem]{Remark}

\theoremstyle{definition}
\newtheorem{definition}[theorem]{Definition}

\numberwithin{equation}{section}

\def\E{\mathbb E}
\def\R{\mathbb R}
\def\M{\mathbb M}

\begin{document}

\title{Two problems related to prescribed  curvature
measures}


\author{Yong Huang}
\address{Wuhan Institute of Physics and Mathematics, Chinese Academy of Sciences, Wuhan 430071,  China}
\email{huangyong@wipm.ac.cn}
\thanks{The  author was supported in part by NSF Grant  11001261.}

%    General info

%\date{}



\begin{abstract}
Existence  of convex body with prescribed generalized curvature
measures is discussed, this result is obtained by making use of
Guan-Li-Li's innovative techniques. In surprise, that methods has
also brought us to promote Ivochkina's $C^2$ estimates for
prescribed curvature equation in \cite{I1, I}.

\end{abstract}

\maketitle

\section{Introduction and main results}
Curvature measure plays fundamental role in the theory of convex
bodies, which is closely related to the differential geometry of
convex hypersurfaces and integral geometry. It has been extensively
studied, see Schneider's book \cite{S}. As Guan, Li and Ma etc.
\cite{GL, GLM1} and their references. Here we give the
interpretation of the problem from point of partial differential
equation for example see \cite{GLM1}. We view $\M$ as a graph over
$\mathbb{S}^n,$ and write $X(x)=\rho(x)x,$ $x\in \mathbb{S}^n,$
$\forall X\in \M.$ Therefore the problem of prescribed (n-k)th
curvature measure can be reduced to the following curvature
equation:
\begin{equation}\label{CM1}
\sigma_k(A)=\frac{f\rho^{1-n}}{\sqrt{\rho^2+|\nabla \rho|^2}},
\end{equation}
where $f>0$ is the given function on $\mathbb{S}^n.$ Moreover,
equation \eqref{CM1} can be expressed as differential equations on
radial function $\rho$ and position vector $X,$
\begin{equation}\label{CM}
\sigma_k(A)=|X|^{-(n+1)}f\left(\frac{X}{|X|}\right)\langle X,
~\nu\rangle\triangleq\phi(X)\langle X, ~\nu\rangle,
\end{equation}
where $\nu$ is the unit outer normal of $\M,$ and
$\lambda=(\lambda_1, \cdots, \lambda_n)$ is the principal curvature
of $M$ at point $X,$
$$
\sigma_k(A)=\sigma_k(\lambda)=\sum\limits_{1\leq i_1<\cdots< i_k\leq
n}\lambda_{i_1} \cdots \lambda_{i_k}.
$$
Alexandrov problem is the zero order curvature measure, which can
also be considered as a counterpart to Minkowski problem. Its
regularity in elliptic case was proved by Pogorelov \cite{P} for
$n=2$ and by Oliker \cite{O} for higher dimension case. The
degenerate case was obtained by Guan-Li \cite{GL2}.

Following ideas from  \cite{CNS, I, Tr1} etc., let us define the
$k-$admissible hypersurfaces:
\begin{definition}
For $1\leq k\leq n,$ let $\Gamma_k$ be a cone in $\R^n$ determined
by
$$
\Gamma_k=\{\lambda\in \R^n: \sigma_l(\lambda)>0,~l=1,2,\cdots,k \}.
$$
A smooth  hypersurface $\M$ is called $k-$admissible if $(\lambda_1,
\lambda_2, \cdots,\lambda_n)\in \Gamma_k$ at every point $X\in\M.$
\end{definition}

There is a difficulty issue around equation \eqref{CM}: the lack of
some appropriate a priori estimates for admissible solutions due to
the appearance of gradient term at right side. That problem has been
open for many years \cite{G}. More recently, Guan-Li-Li\cite{GLL}
have obtained the $C^2$ a priori estimates for the admissible
k-convex starshaped solutions to prescribing (n-k)th curvature
measures for $1\leq k<n.$

 In this paper, we are interested in to consider  the
following problems
\begin{equation}\label{GCM1}
\sigma_k(A)=\langle X, ~\nu\rangle^p\phi(X),
\end{equation}
where $2\leqq k\leqq n.$

Our first motivation is from the existence of convex body with
prescribed curvature measures. Equation \eqref{GCM1} is the equation
of prescribing $(n-k)$th curvature measure for $p=1.$ In particular,
Guan-Li-Li\cite{GLL} has given a open problem for most general
problem in remark 3.5. our result may implies that their conjecture
is correct. Caffarelli-Nirenberg-Spruck\cite{CNS} had been
considered some kind of curvature equation including \eqref{GCM1}
for $p=0$, their $C^1$ estimates depends on barrier conditions.
However, that is not case for our problem. Thus, we consider
\eqref{GCM1}  for $p\neq  0.$ Moreover, we can obtain $C^1$
estimates for a class of curvature equations including \eqref{GCM1}
and quotient curvature equations, its idea is from \cite{GLM1}. Now,
we can state the main theorem.

\begin{theorem}\label{Thm1}
Suppose $\phi(X)$ is a smooth positive function in $\M,$ $ 2\leq
k\leq n£¬,$ $0\neq p\leq 1.$ Then there is a  unique smooth
admissible hypersurface $\M$ satisfying \eqref{GCM1}.
\end{theorem}

Our second motivation is  to generalize Ivochkina's $C^2$ estimates
for prescribed curvature equation in \cite{I1, I} by making use of
those methods. Ivochkina\cite{I1,I} had considered the generalized
type of curvature equation, see also \cite{CNSI, G7, G, SUW, Tr1, U}
and their references,
\begin{equation}\label{mhj}
\sigma_k(A)=\sigma_k(\lambda)=\phi(x, g, Dg),
\end{equation}
where $A$ and $\lambda=(\lambda_1, \cdots, \lambda_n)$denote
respectively the second fundamental form  and  the principle
curvatures of the graph
$$
\M=\{(x, g(x))| x\in \Omega\}.
$$
For doing $C^2$ estimates of \eqref{mhj},  She needed  her condition
(1.5) in \cite{I1}(see also (8.28) in \cite{I}),which is
\begin{equation}\label{hgjh}
k\frac{\partial ^2 \chi^{1/k}}{\partial p^i\partial
p^j}\xi_i\xi_j\geq -\frac{ \chi^{1/k}}{2\sqrt{n}(1+ (\max
|p|)^2)}\xi^2, \quad \xi\in R^n,
\end{equation}
where $\chi(x, g,p)=\phi(x, g,p)(1+|p|^2)^{\frac{k}{2}}.$ We
consider mainly  a kind of  model  from Takimoto\cite{Ta}, which is
also seen as translating solution  of curvature flow.

\begin{equation}\label{Epx}
\sigma_k(\lambda)=\frac{H(x, g)}{(1+|Dg|^2)^{\frac q 2}},
\end{equation}
Ivochkina's conditions (1.5) in \cite{I1} needs $q\leq 0.$ However,
we can generalize Ivochkina's $C^2$ estimate to $q\leq 1$.
\begin{theorem}\label{thmdp}
Suppose $g\in C^4(\Omega)\bigcap C^2(\overline{\Omega})$ is an
admissible  solution of \eqref{Epx} for $q\leq 1$. Then the second
fundamental form $A$ of graph $u$ satisfies
\begin{equation}\label{mk}
\sup\limits_{\Omega}|A|\leq
C\left(1+\sup\limits_{\partial\Omega}|A|\right),
\end{equation}
where $C$ depends only on $n,$ $\|g\|_{C^1(\Omega)},$ and
$\overline{\Omega}\times[\inf\limits_{\partial\Omega}g,~\sup\limits_{\partial\Omega}
g].$
\end{theorem}

This paper is organized as follows: The $C_0$ and $C_1$ bounds and
some elementary formulas were listed in section~\ref{sec2}, the
important $C^2-$estimates are derived in section~\ref{sec4}, which
is by using Guan-Li-Li's innovative methods. In the last section, we
can generalize Ivochkina's
 $C^2$ estimates for prescribed curvature equation in \cite{I1, I}.

\section{Some elementary formals and $C^0$-$C^1$ boundness}
\label{sec2} The standard basis of $\R^{n+1}$ will be denoted by
$\E_1, \E_2,\cdots \E_{n+1},$ and the components of the position
vector $X$ in this basis will be denoted by $X_1, X_2,\cdots,
X_{n+1}.$ We choose an orthonormal frame such that $e_1, e_2,\cdots
e_n$ are tangent to $\M$ and $\nu$ is normal.

% Then in terms of
%$\rho,$ the first fundamental form of $\M$ is given by
%$$
%g_{ij}=\rho^2\delta_{ij}+\nabla_i\nabla_j\rho
%$$

The second fundamental form of $\M$ is given by
\begin{equation}
h_{ij}=\langle D_{e_i}\nu, e_j\rangle,
\end{equation}
and some fundamental formulas are well known for hypersurfaces
$\M\in \R^{n+1}$ as \cite{GLM1}.

\begin{lemma}\label{lemmn}
For any $i,j,l,m=1,\cdots,n$,
\begin{eqnarray}\label{uwg}
&&\nabla_i\nabla_j X=-h_{ij}\nu,\\
&&\nabla_l\nabla_lh_{ij}=\nabla_i\nabla_jh_{ll}+h_{lm}h_{ij}h_{ml}-h_{ll}h_{im}h_{mj},\label{GF} \\
&&\nabla_i \langle X, ~\nu\rangle=h_{il}\langle\nabla_l X,~ X\rangle,\\
&&\nabla_j\nabla_i\langle X, ~\nu\rangle=\nabla_l
h_{ij}\langle\nabla_l X,~ X\rangle +h_{ij}-uh_{im}h_{jm}.
\end{eqnarray}
\end{lemma}

Owing to $\M$ be compact, the $C^0$ estimates is obvious as Lemma
2.2 in \cite{GLM1}.
\begin{lemma}\label{lemc0}
For any compact hypersurface $\M$ satisfying  \eqref{GCM1}, there
are two positive constant $C_1, C_2$ such that
$$
C_1(n, k, \min\limits_{\mathbb{S}^n}
f)\leq\min\limits_{\mathbb{S}^n}|X|\leq \max\limits_{\mathbb{S}^n}
|X|\leq C_2(n, k, \max\limits_{\mathbb{S}^n} f).
$$
\end{lemma}

We have the following gradient estimate for general curvature
measure equation which included \eqref{GCM1} and curvature quotient
equations. As in Guan-Lin-Ma\cite{GLM1}, the result will be obtained
without any barrier condition that was imposed in \cite{CNS}.
Moreover, our result holds for any non zero $p,$ i.e. $0\neq p\in
(-\infty, +\infty).$
\begin{equation}\label{MNB}
F(A)=f(\lambda)=\langle X, ~\nu\rangle^p\phi(X)
\end{equation}
\begin{lemma}\label{lemc1}
If $\M$ satisfies \eqref{MNB} for $F$ with homogeneous of degree
$t>0,$ and for $0\neq p\in (-\infty, +\infty),$ then there exist a
constant $C$ depending only on $n, t,p,$ $\min_{\mathbb{S}^n}\phi,$
$|\phi|_{C^1}$ such that
\begin{equation}
\max\limits_{\mathbb{S}^n}|\nabla \rho|\leq C
\end{equation}
\begin{proof}
We use the method of Guan-Lin-Ma\cite{GLM1}. The gradient bound is
equivalent to $u=\langle X, ~\nu\rangle\geq C>0$ if the lower and
upper bound of the solution holds.

Setting
$$
P(X)=\gamma(\frac{|X|^2}{2})-\log\langle X, ~\nu\rangle,
$$
where the function $\gamma(s)$ is to be determined.

Assume $P(X)$ attains its maximum at point $X_0\in \M.$ We choose
the smooth local orthonormal frame $e_1, \cdots, e_n\in T_{X_0}\M$
such that
$$
\langle X, ~e_i\rangle=0, ~~i\geq 2£¬
$$
Thus£¬ $|X|^2= \langle X, ~e_1\rangle^2+\langle X, ~\nu\rangle^2.$
If $\langle X, ~e_1\rangle^2$ is  also zero, then $|X|^2=\langle X,
~\nu\rangle^2,$ then the bounded from below of  $\langle X,
~\nu\rangle$ is from the bound of $|X|$. We now consider $\langle X,
~e_1\rangle^2>0,$ one has
$$
0=\nabla_i P(X)=\gamma'\langle X, ~e_i\rangle-\frac{h_{im}\langle X,
~e_m\rangle}{\langle X, ~\nu\rangle},
$$
which implies
\begin{equation}
h_{11}=\gamma'\langle X, ~\nu\rangle,\qquad h_{1i}=0, \quad i\geq 2.
\end{equation}
It is easy to know that we only fix  $e_1$ in above process of
choosing local orthonormal frame field $e_1, \cdots, e_n,$ here we
adjust $e_2, \cdots, e_n,$ such that $A=[h_{ij}]$ is diagonal at
$X_0,$ and
\begin{eqnarray*}
0&\geq&F^{ii}\nabla_i\nabla_i P=\gamma''F^{11}\langle X,
~e_1\rangle^2+ \gamma'F^{ii}[\delta_{ii}-h_{ii}\langle X,
~\nu\rangle]\\
&&-\frac{\langle X, ~e_1\rangle}{\langle X, ~\nu\rangle}F^{ii}\nabla_1 h_{ii}
-\frac{F^{ii}h_{ii}}{\langle X, ~\nu\rangle}+F^{ii}h_{ii}^2
+F^{11}\frac{h^2_{11}\langle X, ~e_1\rangle^2}{\langle X, ~\nu\rangle^2}\\
&&=[\gamma''+(\gamma')^2]F^{11}\langle X, ~e_1\rangle^2+
\gamma'F^{ii}[\delta_{ii}-h_{ii}\langle X,
~\nu\rangle]\\
&&-\frac{\langle X, ~e_1\rangle}{\langle X,
~\nu\rangle}F^{ii}\nabla_1 h_{ii}-\frac{F^{ii}h_{ii}}{\langle X,
~\nu\rangle}+F^{ii}h_{ii}^2.
\end{eqnarray*}
Differentiating equation \eqref{MNB} with respect to $e_1,$
\begin{eqnarray*}
F^{ii}\nabla_1h_{ii}&=&p\langle X, ~\nu\rangle^{p-1}\phi \nabla_1 \langle X, ~\nu\rangle+\langle X, ~\nu\rangle^p \nabla_1 \phi\\
&=&p\langle X, ~\nu\rangle^{p-1}\phi h_{11}\langle X,
~e_1\rangle+\langle X, ~\nu\rangle^p \nabla_1 \phi.
\end{eqnarray*}
Thus we obtained
\begin{eqnarray*}
0&\geq&[\gamma''+(\gamma')^2][|X|^2-\langle X, ~\nu\rangle^2]F^{11}+
\gamma'\sum\limits_{i=1}\limits^nF^{ii}-\gamma't\phi\langle X, ~\nu\rangle^{(p+1)}\\
&&-[|X|^2-\langle X, ~\nu\rangle^2]p\phi\langle X,
~\nu\rangle^{(p-1)}\gamma'-\langle X, ~e_1\rangle\langle X, ~\nu\rangle^{(p-1)}\nabla_1\phi\\
&& -t\langle X, ~\nu\rangle^{(p-1)}\phi+F^{ii}h_{ii}^2.
\end{eqnarray*}
So we have
\begin{eqnarray}\label{HY1}
&&[\gamma''+(\gamma')^2]\langle X,
~\nu\rangle^2F^{11}+(t-p)\gamma'\phi\langle X,
~\nu\rangle^{(p+1)}\nonumber \\
&&\quad+\left[|X|^2p\phi\gamma'+\langle X, ~e_1\rangle
\nabla_1\phi+t\phi\right]\langle X,
~\nu\rangle^{(p-1)}\nonumber \\
&&\qquad\geq [\gamma''+(\gamma')^2]|X|^2F^{11}+
\gamma'\sum\limits_{i=1}\limits^nF^{ii}.
\end{eqnarray}
We may assume $\langle X, ~\nu\rangle^2\leq C|X|^2$ for some $C>0,$
otherwise the lemma holds. we claim firstly that
\begin{equation}\label{HY}
(t-p)\gamma'\phi\langle X,
~\nu\rangle^{(p+1)}+\left[|X|^2p\phi\gamma'+\langle X, ~e_1\rangle
\nabla_1\phi+t\phi\right]\langle X, ~\nu\rangle^{(p-1)}\leq 0,
\end{equation}
by taking $\gamma(s)$ properly. We check it by three case of $p<0,$
or $p>t,$ and $0<p\leq t$

We taking
$$
\gamma(s)=\frac{\alpha p}{s},
$$
for $\alpha>0$ is  large enough,

{\bf Case (i) $p<0,$ or $p>t:$} Assuming  $$\langle X,
~\nu\rangle^2\leq \frac{p}{4(p-t)}|X|^2.$$ Then

\begin{eqnarray*}
&&(t-p)\gamma'\phi\langle X, ~\nu\rangle^{(p+1)}+ |X|^2p\phi\gamma'
\langle X, ~\nu\rangle^{(p-1)}\\
&&\qquad=\phi\langle X, ~\nu\rangle^{(p-1)}\left[(t-p)\gamma'\langle
X, ~\nu\rangle^2+p|X|^2\gamma'\right]\\
&&\qquad=-\frac{4\alpha\phi\langle X,
~\nu\rangle^{(p-1)}}{|X|^2}\left[p^2+(t-p)\frac{p\langle X,
~\nu\rangle^2}{|X|^2}\right]\\
&&\qquad\leq -\frac{p^2\alpha\phi\langle X,
~\nu\rangle^{(p-1)}}{|X|^2},
\end{eqnarray*}
which implies \eqref{HY} for $\alpha>0$  large enough.

{\bf Case (ii) $0<p\leq t:$} we have
$$
(t-p)\gamma'\phi\langle X, ~\nu\rangle^{(p+1)}\leq0,
$$
$$
|X|^2p\phi\gamma'+\langle X, ~e_1\rangle
\nabla_1\phi+t\phi=-\frac{4\alpha p^2}{|X|^2}+\langle X, ~e_1\rangle
\nabla_1\phi+t\phi\leq0,
$$
which imply inequality \eqref{HY}.

Combing \eqref{HY1} with \eqref{HY},

\begin{eqnarray}\label{HY2}
&&[\gamma''+(\gamma')^2]\langle X,
~\nu\rangle^2F^{11}\nonumber \\
&&\qquad\geq [\gamma''+(\gamma')^2]|X|^2F^{11}+
\gamma'\sum\limits_{i=1}\limits^nF^{ii}.
\end{eqnarray}
On the other hand,

\begin{eqnarray}\label{HY3}
&&[\gamma''+(\gamma')^2]|X|^2F^{11}+
\gamma'\sum\limits_{i=1}\limits^nF^{ii}\nonumber \\
&&\qquad= [\frac{16\alpha
p}{|X|^6}+\frac{16\alpha^2p^2}{|X|^8}]|X|^2F^{11}-\frac{4\alpha
p}{|X|^4}\sum\limits_{i=1}\limits^nF^{ii}\nonumber \\
&&\qquad \geq C_0F^{11},
\end{eqnarray}
which is from $h_{11}\leq 0$ and then $F^{11}\geq
c_0\sum\limits_{i=1}\limits^nF^{ii}.$ We test the case of
$F=\frac{\sigma_k}{\sigma_l}$($F= \sigma_k $ see \cite{GLM1}), from
(25) in \cite{LT}
$$
F^{11}\geq C(n,l,k)\frac{\sigma_{k-1}(\lambda|1)}{\sigma_l}\geq
C(n,l,k)\frac{\sigma_{k-1}(\lambda )}{\sigma_l},
$$

and
\begin{eqnarray*}
-\sum\limits_{i=1}\limits^nF^{ii}&=&-\frac{(n-k+1)\sigma_{k-1}}{\sigma_l}+\frac{(n-l+1)\sigma_k\sigma_{l-1}}{\sigma_l^2}\\
&\geq&-\frac{(n-k+1)\sigma_{k-1}}{\sigma_l}.
\end{eqnarray*}

Thus $\langle X, ~\nu\rangle\geq C$ is from \eqref{HY2} and
\eqref{HY3}, that is to say there exists a constsnt that depends
only on $n, t,$ $\min_{\mathbb{S}^n}\phi,$ $|\phi|_{C^1}$ such that
$$
\max\limits_{\mathbb{S}^n}|\nabla \rho|\leq C.
$$

 \end{proof}

\end{lemma}

\section{The Important $C^2$ Estimates.}\label{sec4}

The following lemma is key in our proof for $C^2$ estimate, which is
from Guan-Li-Li's important lemma\cite{G1, GLL}.
\begin{lemma}\label{lm1}
For any $\alpha>0,$ one has the following inequality
\begin{eqnarray*}
\sigma_k^{ij,mq}\nabla_sh_{ij}\nabla_s h_{mq}\leq
\sigma_k\left(\frac{\nabla_s\sigma_k}{\sigma_k}-\frac{\nabla_s\sigma_1}{\sigma_1}\right)\left((\alpha+1)\frac{\nabla_s\sigma_k}{\sigma_k}-(\alpha-1)\frac{\nabla_s\sigma_1}{\sigma_1}\right)
\end{eqnarray*}
\end{lemma}
\begin{proof}
From Krylov\cite{K}, for any $\alpha>0,$ $\left(\frac
{\sigma_1}{\sigma_k}\right)^{\alpha}$ is convex on $\Gamma_k,  $
thus,
\begin{eqnarray*}
0&\leq&  \left[\left(\frac
{\sigma_1}{\sigma_k}\right)^{\alpha}\right]^{ij,
mq}\nabla_sh_{ij}\nabla_s h_{mq}\\
&=&\alpha\left(\frac
{\sigma_1}{\sigma_k}\right)^{\alpha-1}2\left(\frac{\sigma_1|\nabla\sigma_k|^2}{\sigma_k^3}-\frac{\langle\nabla\sigma_1, \nabla\sigma_k\rangle}{\sigma_k^2}\right)\\
&&-\alpha\left(\frac {\sigma_1}{\sigma_k}\right)^{\alpha-1}
\frac{\sigma_1}{\sigma_k^2}\sigma_k^{ij,mq}\nabla_sh_{ij}\nabla_s h_{mq}\\
&&+\alpha(\alpha-1)\left(\frac
{\sigma_1}{\sigma_k}\right)^{\alpha-2}\left(\frac{\nabla\sigma_1}{\sigma_k}-\frac{\sigma_1\nabla\sigma_k}{\sigma_k^2}\right)^2.
\end{eqnarray*}
This implies
\begin{eqnarray*}
0&\leq&\frac{2}{\sigma_1}\left(\frac{\sigma_1|\nabla\sigma_k|^2}{\sigma_k}-\langle\nabla\sigma_1, \nabla\sigma_k\rangle\right)\\
&&-\sigma_k^{ij,mq}\nabla_sh_{ij}\nabla_s h_{mq}\\
&&+(\alpha-1)\left(\frac
{\sigma_1^2}{\sigma_k}\right)^{-1}\left(\nabla \sigma_1-\frac{\sigma_1\nabla \sigma_k }{\sigma_k}\right)^2\\
&=&\frac{(\alpha+1)|\nabla \sigma_k|^2}{\sigma_k}-2\alpha\frac{\langle\nabla\sigma_1, \nabla\sigma_k\rangle}{\sigma_1}\\
&&-\sigma_k^{ij,mq}\nabla_sh_{ij}\nabla_s h_{mq}\\
&&+(\alpha-1)\sigma_k\left(\frac{|\nabla\sigma_1|}{\sigma_1}\right)^2.
\end{eqnarray*}
we have proved this lemma.
\end{proof}

\begin{theorem}\label{Thmnb}
Let $\phi(X)$ be a $C^2$ positive function on $\M$, if $\M$ is an
admissible solution of \eqref{GCM1}, we have the following estimates
\begin{equation}
\sigma_1(A)\leq C(n, k, min_{\M}f, \|f\|_{C^2}).
\end{equation}

\end{theorem}
\begin{proof}
Considering

\begin{equation}
F(A)=\sigma_k(A)=\langle X, ~\nu\rangle^p\phi(X),
\end{equation}
we denote $\langle X, ~\nu\rangle$ by $u$ in what follows.

Taking test function $\frac{\sigma_1}{u}$, then at its maximal point
$P$

$$
\nabla_i \left(\ln \frac{\sigma_1}{u}\right)=0,
$$
and

\begin{eqnarray}
0&\geq&F^{ij}\nabla_i \nabla_j\left(\ln \frac{\sigma_1}{u}\right)\nonumber\\
&=&F^{ij}\left[\frac{\nabla_i \nabla_j \sigma_1 }{\sigma_1}-\frac{\nabla_i \sigma_1 \nabla_j\sigma_1 }{\sigma_1^2}-\frac{\nabla_i \nabla_ju }{u}+\frac{\nabla_iu \nabla_j u}{u^2}\right]\nonumber\\
&=&\frac 1 \sigma_1 F^{ij}\nabla_i \nabla_j\sigma_1 -\frac
1uF^{ij}\nabla_i \nabla_ju
\end{eqnarray}

which is equivalent to
\begin{eqnarray}\label{3m}
0&\geq&\frac 1 u F^{ij}\nabla_i \nabla_j
\sigma_1-\-\frac1u\left(\frac{\sigma_1}{u}\right)F^{ij}\nabla_i
\nabla_ju.
\end{eqnarray}

On the other hand, we have

\begin{eqnarray}\label{3m1}
-\frac1u\left(\frac{\sigma_1}{u}\right)F^{ij}\nabla_i \nabla_ju&=&-\frac1u\left(\frac {\sigma_1} u\right)F^{ij}\left[\nabla_mh_{ij }\langle X, X_m\rangle+h_{ij}-h_{im}h_{mj}u\right]\nonumber\\
&=&-\frac 1 u \left(\frac{\sigma_1}{u}\right)\nabla_mF \langle X,
X_m\rangle-kF\frac{\sigma_1}{u^2}+\left(\frac{\sigma_1}{u}\right)F^{ij}h_{im}h_{mj}.
\end{eqnarray}

We also compute the following by lemma~\ref{lemmn},
\begin{eqnarray}\label{3m2}
\frac 1 u F^{ij}\nabla_i \nabla_j\sigma_1&=&\frac 1 u F^{ij}\nabla_s\nabla_sh_{ij}+\frac{kF}{u}|A|^2-\frac1 uF^{ij}h_{im}h_{mj}\sigma_1\nonumber\\
&=&\frac 1 u \triangle F-\frac 1uF^{ij; mq}\nabla_sh_{ij}\nabla_sh_{mq}+\frac{kF}{u}|A|^2-\frac1 uF^{ij}h_{im}h_{mj}\sigma_1\nonumber\\
&=&\frac1 u[u^p\triangle\phi +2pu^{p-1}\langle \nabla \phi, \nabla
u\rangle]+p\phi
u^{p-2}\triangle u +p(p-1)u^{p-3}|\nabla u|^2\nonumber\\
&&-\frac 1uF^{ij; mq}\nabla_sh_{ij}\nabla_sh_{mq}+\frac{kF}{u}|A|^2-\left(\frac{\sigma_1}{u}\right)F^{ij}h_{im}h_{mj}\nonumber\\
&=&\frac1 u[u^p\triangle\phi +2pu^{p-1}\langle \nabla \phi, \nabla
u\rangle]+p\phi
u^{p-2}[\nabla_m \sigma_1 \langle X,~X_m\rangle+g]\nonumber\\
&&+p(p-1)u^{p-3}|\nabla u|^2-\frac 1uF^{ij;
mq}\nabla_sh_{ij}\nabla_sh_{mq}\nonumber\\
&&+\frac{(k-p)F}{u}|A|^2-\left(\frac{\sigma_1}{u}\right)F^{ij}h_{im}h_{mj}.
\end{eqnarray}
Then , combing \eqref{3m} with \eqref{3m1}, \eqref{3m2},
\begin{eqnarray}\label{3q}
0&\geq&-\frac 1uF^{ij; mq}\nabla_sh_{ij}\nabla_sh_{mq}+(k-p)\phi
u^{p-1}|A|^2\nonumber\\
&&+p(p-1)\phi u^{p-3}|\nabla u|^2-C\sigma_1-C.
\end{eqnarray}
Then with lemma~\ref{lm1}, and
$\frac{(\sigma_k)_s}{\sigma_k}=p\frac{u_s}{u}+\frac{\phi_s}{\phi},$
$\frac{(\sigma_1)_s}{\sigma_1}=\frac{u_s}{u},$

\begin{eqnarray*}
0&\geq&(k-p)\phi u^{p-1}|A|^2+(p-1)[p-(\alpha+1)p+(\alpha-1)]\phi
u^{p-3}|\nabla u|^2-C\sigma_1-C,
\end{eqnarray*}
one has the $C^2$ estimate if $(p-1)[p-(\alpha+1)p+(\alpha-1)]\geq
0,$ which is satisfied by taking $p\leq 1$ and
$\alpha=\frac{1}{1-p}>0.$
\end{proof}

From the above certificate process, we know the key point for
proving $C^2$ estimates is the concavity of
$[\frac{\sigma_k}{\sigma_1}]^{\frac 1{k-1}},$ our lemma here has
used the convexity of $\left(\frac
{\sigma_1}{\sigma_k}\right)^{\alpha}$ on $\Gamma_k.$

\section{Proof of theorem~\ref{Thm1}}
\label{secCRT}
\subsection{Proof of theorem~\ref{Thm1}}We use the method
of continuity to prove theorem~\ref{Thm1}. For any positive function
$\phi(X)\in \M,$ and $t\in [0,1],$ setting
$$
\phi_t(X)=1-t+t\phi(X),
$$
we consider a family of equations
\begin{equation}\label{mj}
\sigma_k(A)=\langle X, ~\nu\rangle^p\phi_t(X).
\end{equation}
and
$$
I=\left\{t\in [0,1]| \textit{ Equation } \eqref{mj} \textit{ has a
smoothe admissible solution}\right\}.
$$
For $t=0,$ $X=(C_n^k)^{-\frac{1}{k-p}}$ is a solution of \eqref{mj},
i.e. $I$ is not empty. Moreover, The a prior  estimates
lemma~\ref{lemc0}, lemma~\ref{lemc0} and theorem~\ref{Thmnb} and
Evans-Krylov theorem imply the closeness  of $I.$ we prove the
following proposition that is to say $I$ is open.
\begin{proposition}
Assume $F(\lambda)$ and $\phi(x, \rho, \nabla \rho)$ satisfying
homogeneity property:
\begin{eqnarray}
&& F(\lambda(t\rho))=F(\frac{\lambda(\rho)}{t})\\
&& \phi(x, t\rho, t\nabla \rho)=t^{s}\phi(x, \rho, \nabla \rho).
\end{eqnarray}
Then the linearized operator $L$ of $F(\lambda)=\phi(x, \rho, \nabla
\rho)$ has no non zero kernel, which is from lemma 2.5 in
\cite{GLM1}.
\end{proposition}

Lastly, the uniqueness result of such problem is same as lemma 2.4
in \cite{GLM1}.  We have complete the proof of theorem~\ref{Thm1}.

\section{some discuss about Ivochkina's problem}
Ivochkina\cite{I1,I} had considered the generalized type of
curvature equation, see also \cite{CNSI}v and their references,
\begin{equation}
\sigma_k(\lambda)=\phi(x, g, Dg),
\end{equation}
where $\lambda=(\lambda_1, \cdots, \lambda_n)$ is the principal
curvatures of the graph
$$
\M=\{(x, g(x))| x\in \Omega\}.
$$
In Ivochkina's notation, $\phi(x, g, Dg)=\frac{f(x, g,
Dg)}{(1+|Dg|^2)^{\frac k 2}},$ she need  her condition (1.5) in
\cite{I1}(see also (8.28) in \cite{I}) to do $C^2$ estimate. An
example is
$$
f(x, g, Dg)= H(x, g)(1+|Dg|^2)^s,
$$ for $s\geq \frac{k}{2}.$
Please note that there is a misprint at page 334 in \cite{I1} for
that example, that is no her (1.5) for $s \ge 1/2.$ The author want
to thank for Ivochkina's mention for that.

Here we give a $C^2$ estimate for a special case $\phi(x, u, Du),$
we consider the example which is similar to  (2.6) and (2.7) in
\cite{Ta}.

\begin{equation}\label{Epx}
\sigma_k(A)=\frac{H(x, g)}{(1+|Dg|^2)^{\frac q 2}},
\end{equation}
for this example, Ivochkina need
$$
s=\frac{k-q}{2}\geq \frac k2\Leftrightarrow q\leq 0.
$$
{\bf Proof of Theorem~\ref{thmdp}:}
\begin{proof}
 As \cite{SUW}, taking a local
orthonormal frame field $e_1,\cdots, e_n$ defined on $\M=\{(x,
u(x))| x\in \Omega\}$ in a neighbourhood of the point at which we
are computing and the upward unit normal vector field is
$$
\nu=\frac{(-Dg, 1)}{\sqrt{1+|Dg|^2}}\triangleq \frac{(-Dg, 1)}{w}.
$$
We consider test function $\widehat{W}= w\sigma_1 h(\frac{w^2}{2})$
for $h(t)>0$ to be determined, and it attain its interior maximum at
$X_0.$

By  Lemma 2.1 in \cite{U},
\begin{equation}\label{le}
\nabla_i \nabla_j w=wh_{im}h_{mj}+2\frac{\nabla_i w \nabla_j
w}{w}+w^2\langle \nabla h_{ij},~ E_{n+1}\rangle,
\end{equation}
where $A=[h_{ij}],$ $E_{n+1}$ is the $n+1$-st coordinate vector in
$\R^{n+1}.$
\begin{equation}\label{mbv}
0 = \nabla_i \ln \widehat{W}=\frac{ \nabla_i w}{w}+\frac{
h'w\nabla_i w}{h}+\frac{ \nabla_i \sigma_1}{\sigma_1},
\end{equation}

and
\begin{eqnarray}\label{mnbv}
0&\geq&F^{ij}\nabla_i \nabla_j \ln \widehat{W} \nonumber\\
&=&\left(\frac 1 w+\frac{h'w}{h}\right)F^{ij} \nabla_i \nabla_j w
+F^{ij}\frac{\nabla_i \nabla_j \sigma_1}{\sigma_1} \nonumber\\
&&-\left(\frac{2}{w^2}-\frac{w^2 h''}{h}
+\frac{2h'^2w^2}{h^2}+\frac{h'}{h}\right)F^{ij}\nabla_i w\nabla_j w  \\
&=&\left(1+\frac{h'w^2}{h}\right)F^{ij}h_{mi}h_{mj}
+F^{ij}\frac{\nabla_i \nabla_j \sigma_1}{\sigma_1} \nonumber\\
&&+\left(\frac{w^2 h''}{h}
-\frac{2h'^2w^2}{h^2}+\frac{h'}{h}\right)F^{ij}\nabla_i w\nabla_j w
\nonumber\\
&&+(w+\frac{h'w^3}{h})\langle \nabla  F,~ E_{n+1}\rangle.\nonumber
\end{eqnarray}
By (2.9) in \cite{U},
$$
F^{ij}\frac{\nabla_i \nabla_j \sigma_1}{\sigma_1}=-\frac{F^{ij,
pq}\nabla_\alpha h_{ij}\nabla_\alpha
h_{pq}}{\sigma_1}+F^{ij}h_{ij}\frac{|A|^2}{\sigma_1}-F^{ij}h_{mi}h_{mj}+\frac{\triangle
F}{\sigma_1},
$$
this combines with \eqref{mnbv}, one has

\begin{eqnarray}\label{12bv}
0&\geq&\frac{h'w^2}{h}F^{ij}h_{mi}h_{mj} -\frac{F^{ij,
pq}\nabla_\alpha h_{ij}\nabla_\alpha
h_{pq}}{\sigma_1}+F^{ij}h_{ij}\frac{|A|^2}{\sigma_1}\nonumber\\
&&+\left(\frac{w^2 h''}{h}
-\frac{2h'^2w^2}{h^2}+\frac{h'}{h}\right)F^{ij}\nabla_i w\nabla_j w
\nonumber\\
&&+(w+\frac{h'w^3}{h})\langle \nabla  F,~
E_{n+1}\rangle+\frac{\triangle F}{\sigma_1}.\nonumber
\end{eqnarray}
Taking
$$
h(t)=e^{\frac{t}{2\sup w^2}},
$$
we have
$$
\frac{w^2 h''}{h} -\frac{2h'^2w^2}{h^2}+\frac{h'}{h}\geq 0.
$$
Thus

\begin{eqnarray}\label{waz}
0&\geq&\frac{h'w^2}{h}F^{ij}h_{mi}h_{mj} -\frac{F^{ij,
pq}\nabla_\alpha h_{ij}\nabla_\alpha
h_{pq}}{\sigma_1}+F^{ij}h_{ij}\frac{|A|^2}{\sigma_1}\\
&&+(w+\frac{h'w^3}{h})\langle \nabla  F,~
E_{n+1}\rangle+\frac{\triangle F}{\sigma_1}.\nonumber
\end{eqnarray}
From \eqref{Epx}, we set
$$
F=\sigma_k(A)=\frac{H(x, g)}{(1+|Dg|^2)^{\frac q 2}},
$$
where $H(x, g)$ does not impact the process of proof in the follows,
we may consider the following special case for simplifying  the
denotation,

\begin{equation}\label{xu}
F=\sigma_k(A)=(1+|Dg|^2)^{-\frac q 2}=w^{-q},
\end{equation}
inserting this into \eqref{waz}, and noticing \eqref{le} and
\eqref{mbv},

\begin{eqnarray}\label{wqz}
0&\geq&\frac{h'w^2}{h}F^{ij}h_{mi}h_{mj} -\frac{F^{ij,
pq}\nabla_\alpha h_{ij}\nabla_\alpha
h_{pq}}{\sigma_1}+(k-q)F\frac{|A|^2}{\sigma_1}\\
&&+\frac{(q^2-q)w^{-(2+q)}|\nabla w|^2}{\sigma_1}-C,\nonumber
\end{eqnarray}
for $\sigma_1>> 1.$

On the other hand, lemma~\ref{lm1} and \eqref{xu} implies

\begin{eqnarray}\label{z}
-\frac{F^{ij, pq}\nabla_\alpha h_{ij}\nabla_\alpha
h_{pq}}{\sigma_1}&\geq&\frac{w^{-(2+q)}|\nabla
w|^2}{\sigma_1}\{-(q-1)^2\alpha-q^2+1\nonumber\\
&&+\left[(q-1)(2\alpha-1)+1+q\right]\frac{h'}{h}w^2\\
&&+(1-\alpha)\frac{h'^2}{h^2} w^4\}.\nonumber
\end{eqnarray}
Taking $\alpha=\frac{1}{1-q}, $ if $q<1$  and $0<\alpha<1, $ if
$q=1,$ then
\begin{eqnarray}\label{mw}
-\frac{F^{ij, pq}\nabla_\alpha h_{ij}\nabla_\alpha
h_{pq}}{\sigma_1}+\frac{(q^2-q)w^{-(2+q)}|\nabla w|^2}{\sigma_1}\geq
0.
\end{eqnarray}
Lastly, the $C^2$ estimates  \eqref{mk} is from \eqref{wqz} and
\eqref{mw}.

\end{proof}

\begin{remark}
$\frac{h'w^2}{h}F^{ij}h_{mi}h_{mj}$ is a good term for our estimate
in \eqref{wqz}. Ivochkina has used it in \cite{I1, I}. one may use
it to control the term  like $-\frac{w^{-(2+q)}|\nabla
w|^2}{\sigma_1}$ and refine theorem~\ref{thmdp}.
\end{remark}

\begin{remark}
Takimoto\cite{Ta} had used a priori estimates of the second
derivatives of $u$ in his (2.6) and (2.7) for $1\leq q\leq k-1$ at
page 368. So his result is incomplete.
\end{remark}

\begin{remark}
In the end, an interesting problem is what we can generalize
Ivochkinas' $C^2$ estimates  in \cite{I1, I} to the following
quotient curvature equations?
\begin{equation}\label{LS}
\frac{\sigma_k}{\sigma_l}(A)=H(x, g)w^{-q},
\end{equation}
where $0\leqq l<k\leqq n.$ Of course, it is also interesting for
generalized Guan-Li-Li's results to quotient equations.
\end{remark}

\begin{remark}
More recently, we have found that  Chuanqiang Chen had obtained
$C^2$ estimates for our problem in case of $k=2$ by using different
methods :``Chuanqiang Chen. A Minimal Value Problem and the
Prescribed $\sigma_2$ Curvature Measure Problem, arXiv:1104.4283"
\end{remark}

\section*{Acknowledgments}
The author is grateful to Prof. Pengfei Guan for introducing him to
this problem and for many useful discussions. The work was done
partialy while the author was visiting McGill University in 2009. He
would like to thank their warm hospitality. He also thank for Prof.
Ivochkina's discussion on her paper \cite{I1, I}.

\bibliographystyle{amsplain}

\end{document}